\documentclass[12pt]{amsart}
\usepackage{hyperref}
\usepackage{mathabx}
\usepackage{mathdots}

\newtheorem{theorem}{Theorem}[section]
\newtheorem{definition}[theorem]{Definition}
\newtheorem{proposition}[theorem]{Proposition}
\newtheorem{lemma}[theorem]{Lemma}

\newtheorem{question}[theorem]{Question}

\begin{document}

\title{Counting results for thin Butson matrices}

\author{Teo Banica}
\address{T.B.: Department of Mathematics, Cergy-Pontoise University, 95000 Cergy-Pontoise, France. {\tt teo.banica@gmail.com}}


\begin{abstract}
A partial Butson matrix is a matrix $H\in M_{M\times N}(\mathbb Z_q)$ having its rows pairwise orthogonal, where $\mathbb Z_q\subset\mathbb C^\times$ is the group of $q$-th roots of unity. We investigate here the counting problem for these matrices in the ``thin'' regime, where $M=2,3,\ldots$ is small, and where $N\to\infty$ (subject to the condition $N\in p\mathbb N$ when $q=p^k>2$). The proofs are inspired from the de Launey-Levin and Richmond-Shallit counting results.
\end{abstract}

\maketitle

\section*{Introduction}

A partial Hadamard matrix is a matrix $H\in M_{M\times N}(\pm1)$ having its rows pairwise orthogonal. These matrices are quite interesting objects, appearing in connection with various questions in combinatorics. The motivating examples are the Hadamard matrices $H\in M_N(\pm1)$, and their $M\times N$ submatrices, with $M\leq N$. See \cite{sya}.

A given partial Hadamard matrix $H\in M_{M\times N}(\pm1)$ can complete or not into an Hadamard matrix $\widetilde{H}\in M_N(\pm1)$. It is known since Hall \cite{hal} and Verheiden \cite{ver} that this automatically happens when $K=N-M$ is small, and more precisely when $K\leq7$.

The structure of such matrices is very simple up to $M=4$, where, up to assuming that the first row has 1 entries only, and then permuting the columns, the matrix is:
$$H=\begin{pmatrix}
+&+&+&+&+&+&+&+\\
+&+&+&+&-&-&-&-\\
+&+&-&-&+&+&-&-\\
\underbrace{+}_a&\underbrace{-}_b&\underbrace{+}_b&\underbrace{-}_a&\underbrace{+}_b&\underbrace{-}_a&\underbrace{+}_a&\underbrace{-}_b
\end{pmatrix}$$

Here $a,b\in\mathbb N$ are subject to the condition $a+b=N/4$.

At $M\geq5$ no such result is available, and the partial Hadamard matrices give rise to interesting combinatorial structures, related to the Hadamard Conjecture. See Ito \cite{ito}.

In their breakthrough paper \cite{ll2}, following some previous work in \cite{ll1}, de Launey and Levin proposed a whole new point of view on these matrices, in the asymptotic limit $N\in 4\mathbb N$, $N\to\infty$. Their main result is as follows:

\medskip

\noindent {\bf Theorem (de Launey-Levin \cite{ll2}).} {\em The probability for a random $H\in M_{M\times N}(\pm1)$ to be partial Hadamard is
$$P_M\simeq\frac{2^{(M-1)^2}}{\sqrt{(2\pi N)^{\binom{M}{2}}}}$$
in the $N\in 4\mathbb N$, $N\to\infty$ limit.}

\medskip

The proof in \cite{ll2} uses a random walk interpretation of the partial Hadamard matrices, then the Fourier inversion formula, and then some real analysis methods. Importantly, as pointed out there, this method can be probably used for more general situations.

An interesting generalization of the Hadamard matrices are the complex Hadamard matrices $H\in M_N(\mathbb C)$ having as entries the roots of unity, introduced by Butson in \cite{but}. The basic example here is the Fourier matrix, $F_N=(w^{ij})$ with $w=e^{2\pi i/N}$:
$$F_N=\begin{pmatrix}
1&1&1&\ldots&1\\
1&w&w^2&\ldots&w^{N-1}\\
\ldots&\ldots&\ldots&\ldots&\ldots\\
1&w^{N-1}&w^{2(N-1)}&\ldots&w^{(N-1)^2}
\end{pmatrix}$$

In general, the theory of Butson matrices can be regarded as a ``non-standard'' branch of discrete Fourier analysis. For a number of results on these matrices, see \cite{tzy}.

We can of course talk about partial Buston matrices:

\medskip

\noindent {\bf Definition.} {\em A partial Butson matrix is a matrix $H\in M_{M\times N}(\mathbb Z_q)$ having its rows pairwise orthogonal, where $\mathbb Z_q\subset\mathbb C^\times$ is the group of $q$-roots of unity.}

\medskip

Observe that at $q=2$ we obtain the partial Hadamard matrices. In general, the interest comes from the Butson matrices $H\in M_N(\mathbb Z_q)$, and from their $M\times N$ submatrices.

Let us first discuss the case $q=2^k$. At $M=2$, up to assuming that the first row has 1 entries only, and then permuting the columns, the matrix must be as follows:
$$H=\begin{pmatrix}
1&1&\ldots&1&1&1&\ldots&1\\
\underbrace{1}_{a_1}&\underbrace{w}_{a_2}&\ldots&\underbrace{w^{q/2-1}}_{a_{q/2}}&\underbrace{w^{q/2}}_{a_1}&\underbrace{w^{q/2+1}}_{a_2}&\ldots&\underbrace{w^{q-1}}_{a_{q/2}}
\end{pmatrix}$$

Here $w=e^{2\pi i/q}$ and $a_1,\ldots,a_{q/2}\in\mathbb N$ are certain multiplicities, summing up to $N/2$. Thus counting such objects is the same as counting abelian squares, i.e. length $N$ words of type $xx'$ where $x'$ is a permutation of $x$.  According now to \cite{rsh}, we have:

\medskip

\noindent {\bf Theorem (cf. Richmond-Shallit \cite{rsh}).} {\em For $q=2^k$ the probability for a randomly chosen $H\in M_{2\times N}(\mathbb Z_q)$ to be partial Butson is
$$P_2\simeq2\sqrt{\left(\frac{q/2}{2\pi N}\right)^{q/2}}$$
in the $N\in 2\mathbb N$, $N\to\infty$ limit.}

\medskip

There are actually several proofs of this result, but the one in \cite{rsh} is remarkably beautiful: based only on the Stirling formula, and on an old idea of Lagrange. Indeed:
$$P_2=\frac{1}{q^N}\binom{N}{N/2}\sum_{a_1+\ldots+a_{q/2}=N/2}\binom{N/2}{a_1,\ldots,a_{q/2}}^2$$

The point now is that the sum on the right can be estimated by making a clever use of the Stirling formula, and this gives the above result. See \cite{rsh}.

Summarizing, there are several techniques for dealing with the counting problem for partial Butson matrices. In this paper we will try to use and mix these techniques. Our first result here will be an extension of the Richmond-Shallit count:

\medskip

\noindent {\bf Theorem A.} {\em When $q=p^k$ is a prime power, the probability for a randomly chosen $H\in M_{2\times N}(\mathbb Z_q)$, with $N\in p\mathbb N$, $N\to\infty$, to be partial Butson is:
$$P_2\simeq\sqrt{\frac{p^{2-\frac{q}{p}}q^{q-\frac{q}{p}}}{(2\pi N)^{q-\frac{q}{p}}}}$$
In particular, for $q=p$ prime, $P_2\simeq\sqrt{\frac{p^p}{(2\pi N)^{p-1}}}$.}

\medskip

When $q\in\mathbb N$ is not a prime power the combinatorics is much more complicated, as shown by Lam and Leung in \cite{lle}. Particularly problematic is the case where $q$ has 3 prime factors, because the vanishing sums of $q$-roots of unity won't necessarily decompose as sums of cycles. Here is such a ``tricky'' vanishing sum, with $w=e^{2\pi i/30}$:
$$w^5+w^6+w^{12}+w^{18}+w^{24}+w^{25}=0$$

Our second result will concern the case where $q$ has two prime factors. If we call ``dephased'' the matrices having the first row consisting of 1 entries only, we have:

\medskip

\noindent {\bf Theorem B.} {\em For $q=p_1^{k_1}p_2^{k_2}$ with $p_1,p_2$ distinct primes, the dephased partial Butson matrices $H\in M_{2\times N}(\mathbb Z_q)$ are indexed by matrices $A\in M_{p_1^{k_1}\times p_2^{k_2}}(\mathbb N)$ of the following form, with indices $i\in\mathbb Z_{p_1}$, $j\in\mathbb Z_{p_1^{k_1-1}}$, $x\in\mathbb Z_{p_2}$, $y\in\mathbb Z_{p_2^{k_2-1}}$, and with $B_{ijy},C_{jxy}\in\mathbb N$:
$$A_{ij,xy}=B_{ijy}+C_{jxy}$$
In particular at $q=2p$ with $p\geq3$ prime, $P_2$ equals the probability for a random walk on $\mathbb Z^p$ to end up on the diagonal, i.e. at a position of type $(t,\ldots,t)$, with $t\in\mathbb Z$.}

\medskip

As already mentioned, the general case $q=p_1^{k_1}\ldots p_s^{k_s}$ is certainly more complicated. One way of avoiding the difficulties would be by imposing the ``regularity'' assumption from \cite{bbs}. But the matrices $A$ as above will become $s$-arrays, and we have no results.

Finally, at $M=3$, and when $q=p$ is prime, the partial Butson matrices are related to the matrices  $A\in M_p(\mathbb N)$ which are ``tristochastic'', in the sense that the sums on the rows, columns and diagonals are all equal. We will prove the following result:

\medskip

\noindent {\bf Theorem C.} {\em At $q=p$ prime, the dephased partial Butson matrices $H\in M_{3\times N}(\mathbb Z_q)$ are indexed by the tristochastic matrices $A\in M_p(\mathbb N)$, with sum $N/p$. In particular at $p=3$ we have $P_3\simeq\frac{243\sqrt{3}}{(2\pi N)^3}$, in the $N\in 3\mathbb N$, $N\to\infty$ limit.}

\medskip

We can see from the above results that the counting problem depends a lot on $q,M$. We believe that an extension of \cite{ll2} should require assuming that $q=p$ is prime.

The paper is organized as follows: 1 is a preliminary section, in 2-4 we state and prove our main results, and 5 contains some further results, and a few concluding remarks.

\section{Partial Hadamard matrices}

Let $H\in M_{M\times N}(\pm1)$ be a partial Hadamard matrix (PHM). We will usually dephase $H$, i.e. assume that the first row consists of $1$ entries only, then put it in ``standard form'', with the $+$ entries moved to the left as much as possible, by proceeding from top to bottom. Here are some examples, at small values of $M$:

\begin{proposition}
The standard form of dephased PHM at $M=2,3,4$ is
$$H=\begin{pmatrix}+&+\\\underbrace{+}_{N/2}&\underbrace{-}_{N/2}\end{pmatrix}\quad 
H=\begin{pmatrix}+&+&+&+\\+&+&-&-\\\underbrace{+}_{N/4}&\underbrace{-}_{N/4}&\underbrace{+}_{N/4}&\underbrace{-}_{N/4}\end{pmatrix}$$
$$H=\begin{pmatrix}
+&+&+&+&+&+&+&+\\
+&+&+&+&-&-&-&-\\
+&+&-&-&+&+&-&-\\
\underbrace{+}_a&\underbrace{-}_b&\underbrace{+}_b&\underbrace{-}_a&\underbrace{+}_b&\underbrace{-}_a&\underbrace{+}_a&\underbrace{-}_b
\end{pmatrix}$$
where at $M=4$ the numbers $a,b\in\mathbb N$ satisfy $a+b=N/4$.
\end{proposition}

\begin{proof}
All the results follow by putting the matrix in standard form, and then writing down the orthogonality equations in terms of the block entries in the last row.
\end{proof}

Let us try now to count the partial Hadamard matrices $H\in M_{M\times N}(\pm1)$. This is an easy task at $M=2,3,4$, where the answer is:

\begin{proposition}
The number of PHM at $M=2,3,4$ is
\begin{eqnarray*}
\#PHM_{2\times N}&=&2^N\binom{N}{N/2}\\
\#PHM_{3\times N}&=&2^N\binom{N}{N/4,N/4,N/4,N/4}\\
\#PHM_{4\times N}&=&2^N\sum_{a+b=N/4}\binom{N}{a,b,b,a,b,a,a,b}
\end{eqnarray*}
where the quantities on the right are multinomial coefficients.
\end{proposition}

\begin{proof}
Indeed, the multinomial coefficients at right count the matrices having the first row consisting of 1 entries only, and the $2^N$ factor comes from this.
\end{proof}

At $M\geq 5$ no such simple formula is available, and estimating rather than exactly computing looks like a more reasonable objective. First, we have:

\begin{proposition}
The probability for a random $H\in M_{M\times N}(\pm1)$ to be PHM is
$$P_2\simeq\frac{2}{\sqrt{2\pi N}},\quad P_3\simeq\frac{16}{\sqrt{(2\pi N)^3}},\quad P_4\simeq\frac{512}{(2\pi N)^3}$$
in the $N\in2\mathbb N$ (resp. $N\in4\mathbb N$, $N\in4\mathbb N$), $N\to\infty$ limit.
\end{proposition}

\begin{proof}
Since there are $2^{MN}$ sign matrices of size $N\times M$, the probability $P_M$ in the statement is given by:
$$P_M=\frac{1}{2^{MN}}\#PHM_{M\times N}$$

With this formula in hand, the result follows from Proposition 1.2, by using standard estimates for sums of binomial coefficients (see Lemma 2.4 below). 
\end{proof}

In general, we have the following result, due to de Launey and Levin:

\begin{theorem}[\cite{ll2}]
The probability for a random $H\in M_{M\times N}(\pm1)$ to be PHM is
$$P_M\simeq\frac{2^{(M-1)^2}}{\sqrt{(2\pi N)^{\binom{M}{2}}}}$$
in the $N\in 4\mathbb N$, $N\to\infty$ limit.
\end{theorem}

\begin{proof}
The proof in \cite{ll2} uses a random walk interpretation of the PHM, then the Fourier inversion formula, and finally a number of quite technical real analysis estimates.
\end{proof}

\section{Butson matrices, abelian squares}

As mentioned in \cite{ll2}, the method there should apply to more general situations. We discuss in what follows a potential extension to the partial Butson matrices:

\begin{definition}
A partial Butson matrix (PBM) is a matrix $H\in M_{M\times N}(\mathbb Z_q)$ having its rows pairwise orthogonal, where $\mathbb Z_q\subset\mathbb C^\times$ is the group of $q$-roots of unity.
\end{definition}

Observe that at $q=2$ we obtain the PHM. In general, the interest comes from the Butson matrices $H\in M_N(\mathbb Z_q)$, and from their $M\times N$ submatrices. See \cite{but}, \cite{tzy}.

Two PBM are called ``equivalent'' if one can pass from one to the other by permuting the rows and columns, or by multiplying the rows and columns by numbers in $\mathbb Z_q$.

Up to this equivalence, we can assume that $H$ is dephased, in the sense that its first row consists of $1$ entries only. We can also put $H$ in ``standard form'', as follows:

\begin{definition}
We say that $H\in M_{M\times N}(\mathbb Z_q)$ is in standard form if the low powers of $w=e^{2\pi i/q}$ are moved to the left as much as possible, by proceeding from top to bottom.
\end{definition}

Let us first try to understand the case $M=2$. Here a dephased partial Butson matrix $H\in M_{2\times N}(\mathbb Z_q)$ must look as follows, with $\lambda_i\in\mathbb Z_q$ satisfying $\lambda_1+\ldots+\lambda_N=0$:
$$H=\begin{pmatrix}1&\ldots&1\\ \lambda_1&\ldots&\lambda_N\end{pmatrix}$$

With $q=p_1^{k_1}\ldots p_s^{k_s}$, we must have, according to Lam and Leung \cite{lle}, $N\in p_1\mathbb N+\ldots+p_s\mathbb N$. Observe however that at $s\geq 2$ this obstruction dissapears at $N\geq p_1p_2$.

In this section we restrict attention to the prime power case. First, we have:

\begin{proposition}
When $q=p^k$ is a prime power, the standard form of the dephased partial Butson matrices at $M=2$ is
$$H=\begin{pmatrix}
1&1&\ldots&1&\ldots&\ldots&1&1&\ldots&1\\
\underbrace{1}_{a_1}&\underbrace{w}_{a_2}&\ldots&\underbrace{w^{q/p-1}}_{a_{q/p}}&\ldots&\ldots&\underbrace{w^{q-q/p}}_{a_1}&\underbrace{w^{q-q/p+1}}_{a_2}&\ldots&\underbrace{w^{q-1}}_{a_{q/p}}
\end{pmatrix}$$
where $w=e^{2\pi i/q}$ and where $a_1,\ldots,a_{q/p}\in\mathbb N$ are multiplicities, summing up to $N/p$.
\end{proposition}

\begin{proof}
Indeed, it is well-known that for $q=p^k$ the solutions of $\lambda_1+\ldots+\lambda_N=0$ with $\lambda_i\in\mathbb Z_q$ are, up to permutations of the terms, exactly those in the statement.
\end{proof}

Our next objective will be to count the matrices in Proposition 2.3. We use:

\begin{lemma}
We have the estimate
$$\sum_{a_1+\ldots+a_s=n}\binom{n}{a_1,\ldots,a_s}^p
\simeq s^{pn}\sqrt{\frac{s^{s(p-1)}}{p^{s-1}(2\pi n)^{(s-1)(p-1)}}}$$
in the $n\to\infty$ limit.
\end{lemma}

\begin{proof}
This is proved by Richmond and Shallit in \cite{rsh} at $p=2$, and the proof in the general case, $p\in\mathbb N$, is similar. More precisely, let us denote by $c_{sp}$ the sum on the left. By setting $a_i=\frac{n}{s}+x_i\sqrt{n}$ and then by using the various formulae in \cite{rsh}, we obtain:
\begin{eqnarray*}
c_{sp}
&\simeq&s^{pn}(2\pi n)^{\frac{(1-s)p}{2}}s^{\frac{sp}{2}}\exp\left(-\frac{sp}{2}\sum_{i=1}^sx_i^2\right)\\
&\simeq&s^{pn}(2\pi n)^{\frac{(1-s)p}{2}}s^{\frac{sp}{2}}\underbrace{\int_0^n\ldots\int_0^n}_{s-1}\exp\left(-\frac{sp}{2}\sum_{i=1}^sx_i^2\right)da_1\ldots da_{s-1}\\
&=&s^{pn}(2\pi n)^{\frac{(1-s)p}{2}}s^{\frac{sp}{2}}n^{\frac{s-1}{2}}\underbrace{\int_0^n\ldots\int_0^n}_{s-1}\exp\left(-\frac{sp}{2}\sum_{i=1}^{s-1}x_i^2-\frac{sp}{2}\left(\sum_{i=1}^{s-1}x_i\right)^2\right)dx_1\ldots dx_{s-1}\\
&=&s^{pn}(2\pi n)^{\frac{(1-s)p}{2}}s^{\frac{sp}{2}}n^{\frac{s-1}{2}}\times\pi^{\frac{s-1}{2}}s^{-\frac{1}{2}}\left(\frac{sp}{2}\right)^{\frac{1-s}{2}}\\
&=&s^{pn}(2\pi n)^{\frac{(1-s)p}{2}}s^{\frac{sp}{2}-\frac{1}{2}+\frac{1-s}{2}}\left(\frac{p}{2\pi n}\right)^{\frac{1-s}{2}}\\
&=&s^{pn}(2\pi n)^{\frac{(1-s)(p-1)}{2}}s^{\frac{sp-s}{2}}p^{\frac{1-s}{2}}
\end{eqnarray*}

Thus we have obtained the formula in the statement, and we are done.
\end{proof}

Now with Lemma 2.4 in hand, we can now prove:

\begin{theorem}
When $q=p^k$ is a prime power, the probability for a randomly chosen $M\in M_{2\times N}(\mathbb Z_q)$, with $N\in p\mathbb N$, $N\to\infty$, to be partial Butson is:
$$P_2\simeq\sqrt{\frac{p^{2-\frac{q}{p}}q^{q-\frac{q}{p}}}{(2\pi N)^{q-\frac{q}{p}}}}$$
In particular, for $q=p$ prime, $P_2\simeq\sqrt{\frac{p^p}{(2\pi N)^{p-1}}}$. Also, for $q=2^k$, $P_2\simeq2\sqrt{\left(\frac{q/2}{2\pi N}\right)^{q/2}}$.
\end{theorem}

\begin{proof}
First, the probability $P_M$ for a random $M\in M_{M\times N}(\mathbb Z_q)$ to be PBM is:
$$P_M=\frac{1}{q^{MN}}\#PBM_{M\times N}$$

Thus, according to Proposition 2.3, we have the following formula:
\begin{eqnarray*}
P_2
&=&\frac{1}{q^N}\sum_{a_1+\ldots +a_{q/p}=N/p}\binom{N}{\underbrace{a_1\ldots a_1}_p\ldots\ldots\underbrace{a_{q/p}\ldots a_{q/p}}_p}\\
&=&\frac{1}{q^N}\binom{N}{\underbrace{N/p\ldots N/p}_p}\sum_{a_1+\ldots +a_{q/p}=N/p}\binom{N/p}{a_1\ldots a_{q/p}}^p\\
&=&\frac{1}{p^N}\binom{N}{\underbrace{N/p\ldots N/p}_p}\times\frac{1}{(q/p)^N}\sum_{a_1+\ldots +a_{q/p}=N/p}\binom{N/p}{a_1\ldots a_{q/p}}^p
\end{eqnarray*}

Now by using the Stirling formula for the left term, and Lemma 2.4 with $s=q/p$ and $n=N/p$ for the right term, we obtain:
\begin{eqnarray*}
P_2
&=&\sqrt{\frac{p^p}{(2\pi N)^{p-1}}}\times\sqrt{\frac{(q/p)^{\frac{q}{p}(p-1)}}{p^{\frac{q}{p}-1}(2\pi N/p)^{(\frac{q}{p}-1)(p-1)}}}\\
&=&\sqrt{\frac{p^{p-\frac{q}{p}(p-1)-\frac{q}{p}+1+(\frac{q}{p}-1)(p-1)}q^{\frac{q}{p}(p-1)}}{(2\pi N)^{p-1+(\frac{q}{p}-1)(p-1)}}}\\
&=&\sqrt{\frac{p^{2-\frac{q}{p}}q^{q-\frac{q}{p}}}{(2\pi N)^{q-\frac{q}{p}}}}
\end{eqnarray*}

Thus we have obtained the formula in the statement, and we are done.
\end{proof}

\section{Two prime factors, random walks}

In this section we discuss the case where $M=2$ and $q=p_1^{k_1}p_2^{k_2}$ has two prime factors. Let us first examine the simplest such case, namely $q=p_1p_2$, with $p_1,p_2$ primes:

\begin{proposition}
When $q=p_1p_2$ is a product of distinct primes, the standard form of the dephased partial Butson matrices at $M=2$ is
$$H=\begin{pmatrix}
1&1&\ldots&1&\ldots&\ldots&1&1&\ldots&1\\
\underbrace{1}_{A_{11}}&\underbrace{w}_{A_{12}}&\ldots&\underbrace{w^{p_2-1}}_{A_{1p_2}}&\ldots&\ldots&\underbrace{w^{q-p_2}}_{A_{p_11}}&\underbrace{w^{q-p_2+1}}_{A_{p_12}}&\ldots&\underbrace{w^{q-1}}_{A_{p_1p_2}}
\end{pmatrix}$$
where $w=e^{2\pi i/q}$, and $A\in M_{p_1\times p_2}(\mathbb N)$ is of the form $A_{ij}=B_i+C_j$, with $B_i,C_j\in\mathbb N$.
\end{proposition}

\begin{proof}
We use the fact that for $q=p_1p_2$ any vanishing sum of $q$-roots of unity decomposes as a sum of cycles. See \cite{lle}. Now if we denote by $B_i,C_j\in\mathbb N$ the multiplicities of the various $p_2$-cycles and $p_1$-cycles, then we must have $A_{ij}=B_i+C_j$, as claimed.
\end{proof}

Regarding the matrices of type $A_{ij}=B_i+C_j$, when taking them over integers, $B_i,C_j\in\mathbb Z$, these form a vector space of dimension $p_1+p_2-1$. Given $A\in M_{p_1\times p_2}(\mathbb Z)$, the ``test'' for deciding if we have  $A_{ij}=B_i+C_j$ or not is $A_{ij}+A_{kl}=A_{il}+A_{jk}$.

The problem comes of course from the assumption $B_i,C_j\geq0$, which is quite a subtle one. In what follows we restrict attention to the case $p_1=2$. Here we have:

\begin{theorem}
For $q=2p$ with $p\geq 3$ prime, $P_2$ equals the probability for a random walk on $\mathbb Z^p$ to end up on the diagonal, i.e. at a position of type $(t,\ldots,t)$, with $t\in\mathbb Z$.
\end{theorem}

\begin{proof}
According to Proposition 3.1, we must understand the matrices $A\in M_{2\times p}(\mathbb N)$ which decompose as $A_{ij}=B_i+C_j$, with $B_i,C_j\geq0$. But this is an easy task, because depending on $A_{11}$ vs. $A_{21}$ we have 3 types of solutions, as follows:
$$\begin{pmatrix}
a_1&\ldots&a_p\\
a_1&\ldots&a_p
\end{pmatrix},\qquad
\begin{pmatrix}
a_1&\ldots&a_p\\
a_1+t&\ldots&a_p+t
\end{pmatrix},\qquad
\begin{pmatrix}
a_1+t&\ldots&a_p+t\\
a_1&\ldots&a_p
\end{pmatrix}$$

Here $a_i\geq0$ and $t\geq1$. Now since cases 2,3 contribute in the same way, we obtain:
\begin{eqnarray*}
P_2
&=&\frac{1}{(2p)^N}\sum_{2\Sigma a_i=N}\binom{N}{a_1,a_1,\ldots,a_p,a_p}\\
&+&\frac{2}{(2p)^N}\sum_{t\geq1}\sum_{2\Sigma a_i+pt=N}\binom{N}{a_1,a_1+t,\ldots,a_p,a_p+t}
\end{eqnarray*}

We can write this formula in a more compact way, as follows:
$$P_2=\frac{1}{(2p)^N}\sum_{t\in\mathbb Z}\sum_{2\Sigma a_i+p|t|=N}\binom{N}{a_1,a_1+|t|,\ldots,a_p,a_p+|t|}$$

Now since the sum on the right, when rescaled by $\frac{1}{(2p)^N}$, is exactly the probability for a random walk on $\mathbb Z^p$ to end up at $(t,\ldots,t)$, this gives the result.
\end{proof}

According to the above result we have $P_2=\sum_{t\in\mathbb Z}P_2^{(t)}$, where $P_2^{(t)}$ with $t\in\mathbb Z$ is the probability for a random walk on $\mathbb Z^p$ to end up at $(t,\ldots,t)$. Observe that, by using Lemma 2.4 above with $s,p,n$ equal respectively to $p,2,N/2$, we obtain:
\begin{eqnarray*}
P_2^{(0)}
&=&\frac{1}{(2p)^N}\binom{N}{N/2}\sum_{a_1+\ldots+a_p=N/2}\binom{N/2}{a_1,\ldots,a_p}^2\\
&\simeq&\sqrt{\frac{2}{\pi N}}\times\sqrt{\frac{p^p}{2^{p-1}(\pi N)^{p-1}}}
=2\sqrt{\left(\frac{p}{2\pi N}\right)^p}
\end{eqnarray*}

Regarding now the probability $P_2^{(t)}$ of ending up at $(t,\ldots,t)$, in principle for small $t$ this can be estimated by using a modification of the method in \cite{rsh}. However, it is not clear on how to compute the full diagonal return probability in Theorem 3.2.

Let us discuss now the exponents $q=3p$. The same method as in the proof of Theorem 3.2 works, with the ``generic'' solution for $A$ being as follows:
$$A=\begin{pmatrix}
a_1&\ldots&a_p\\
a_1+t&\ldots&a_p+t\\
a_1+s+t&\ldots&a_p+s+t\\
\end{pmatrix}$$

More precisely, this type of solution, with $s,t\geq1$, must be counted 6 times, then its $s=0,t\geq1$ and $s\geq1,t=0$ particular cases must be counted 3 times each, and finally the $s=t=0$ case must be counted once. Observe that the $s=t=0$ contribution is:
\begin{eqnarray*}
P_3^{(0,0)}
&=&\frac{1}{(3p)^N}\binom{N}{N/3,N/3,N/3}\sum_{a_1+\ldots+a_p=N/3}\binom{N/3}{a_1,\ldots,a_p}^3\\
&\simeq&\sqrt{\frac{27}{(2\pi N)^2}}\times\sqrt{\frac{p^{2p}}{3^{p-1}(2\pi N/3)^{2(p-1)}}}\\
&=&3\sqrt{3^p}\left(\frac{p}{2\pi N}\right)^p
\end{eqnarray*} 

Finally, regarding arbitrary exponents with two prime factors, we have:

\begin{proposition}
When $q=p_1^{k_1}p_2^{k_2}$ has exactly two prime factors, the dephased partial Butson matrices at $M=2$ are indexed by the solutions of
$$A_{ij,xy}=B_{ijy}+C_{jxy}$$
with $B_{ijy},C_{jxy}\in\mathbb N$, with $i\in\mathbb Z_{p_1}$, $j\in\mathbb Z_{p_1^{k_1-1}}$, $x\in\mathbb Z_{p_2}$, $y\in\mathbb Z_{p_2^{k_2-1}}$.
\end{proposition}

\begin{proof}
We follow the method in the proof of Proposition 3.1. First, according to \cite{lle}, for $q=p_1^{k_1}p_2^{k_2}$ any vanishing sum of $q$-roots of unity decomposes as a sum of cycles.

Let us first work out a simple particular case, namely $q=4p$. Here the multiplicity matrices $A\in M_{4\times p}(\mathbb N)$ appear as follows:
$$A=\begin{pmatrix}B_1&\ldots&B_1\\ B_2&\ldots&B_2\\ B_3&\ldots&B_3\\ B_4&\ldots&B_4\end{pmatrix}+
\begin{pmatrix}C_1&\ldots&C_p\\ D_1&\ldots&D_p\\ C_1&\ldots&C_p\\ D_1&\ldots&D_p\end{pmatrix}$$

Thus, if we use double binary indices for the elements of $\{1,2,3,4\}$, the condition is:
$$A_{ij,x}=B_{ij}+C_{jx}$$

The same method works for any exponent of type $q=p_1^{k_1}p_2^{k_2}$, the formula being:
$$A_{i_1\ldots i_{k_1},x_1\ldots x_{k_2}}=B_{i_1\ldots i_{k_1},x_2\ldots x_{k_2}}+C_{i_2\ldots i_{k_1},x_1\ldots x_{k_2}}$$

But this gives the formula in the statement, and we are done.
\end{proof}

\section{Three rows: tristochastic matrices}

At $M=3$ now, we first restrict attention to the case where $q=p$ is prime. In this case, Proposition 2.3 becomes simply:
$$H=\begin{pmatrix}
1&1&\ldots&1\\
\underbrace{1}_a&\underbrace{w}_a&\ldots&\underbrace{w^{p-1}}_a
\end{pmatrix}$$

We call a matrix $A\in M_p(\mathbb N)$ ``tristochastic'' if the sums on its rows, columns and diagonals are all equal. Here, and in what follows, we call ``diagonals'' the main diagonal, and its $p-1$ translates to the right, obtained by using modulo $p$ indices.

With this notation, here is now the result at $M=3$:

\begin{proposition}
For $p$ prime, the standard form of the dephased PBM at $M=3$ is
$$H=\begin{pmatrix}
1&1&\ldots&1&\ldots&\ldots&1&1&\ldots&1\\
1&1&\ldots&1&\ldots&\ldots&w^{p-1}&w^{p-1}&\ldots&w^{p-1}\\
\underbrace{1}_{A_{11}}&\underbrace{w}_{A_{12}}&\ldots&\underbrace{w^{p-1}}_{A_{1p}}&\ldots&\ldots&\underbrace{1}_{A_{p1}}&\underbrace{w}_{A_{p2}}&\ldots&\underbrace{w^{p-1}}_{A_{pp}}
\end{pmatrix}$$
where $w=e^{2\pi i/p}$ and where $A\in M_p(\mathbb N)$ is tristochastic, with sums $N/p$.
\end{proposition}

\begin{proof}
Consider a dephased matrix $H\in M_{3\times N}(\mathbb Z_p)$, written in standard form as in the statement. Then the orthogonality conditions between the rows are as follows:

$1\perp 2$ means $A_{11}+\ldots+A_{1p}=A_{21}+\ldots+A_{2p}=\ldots\ldots=A_{p1}+\ldots+A_{pp}$.

$1\perp 3$ means $A_{11}+\ldots+A_{p1}=A_{12}+\ldots+A_{p2}=\ldots\ldots=A_{1p}+\ldots+A_{pp}$.

$2\perp 3$ means $A_{11}+\ldots+A_{pp}=A_{12}+\ldots+A_{p1}=\ldots\ldots=A_{1p}+\ldots+A_{p,p-1}$.

Thus $A$ must have constant sums on rows, columns and diagonals, as claimed.
\end{proof}

It is quite unobvious on how to deal with the tristochastic matrices with bare hands. For the moment, let us just record a few elementary results:

\begin{proposition}
For $p=2,3$, the standard form of the dephased PBM at $M=3$ is respectively as follows, with $w=e^{2\pi i/3}$ and $a+b+c=N/3$ at $p=3$:
$$H=\begin{pmatrix}+&+&+&+\\+&+&-&-\\\underbrace{+}_{N/4}&\underbrace{-}_{N/4}&\underbrace{+}_{N/4}&\underbrace{-}_{N/4}\end{pmatrix}$$
$$H=\begin{pmatrix}
1&1&1&1&1&1&1&1&1\\
1&1&1&w&w&w&w^2&w^2&w^2\\
\underbrace{1}_a&\underbrace{w}_b&\underbrace{w^2}_c&\underbrace{1}_b&\underbrace{w}_c&\underbrace{w^2}_a&\underbrace{1}_c&\underbrace{w}_a&\underbrace{w^2}_b
\end{pmatrix}$$
Also, for $p\geq 3$ prime and $N\in p\mathbb N$, there is at least one Butson matrix $H\in M_{3\times N}(\mathbb Z_p)$.
\end{proposition}

\begin{proof}
The $p=2,3$ assertions follow from Proposition 4.1, and from the fact that the $2\times 2$ and $3\times 3$ tristochastic matrices are respectively as follows:
$$A=\begin{pmatrix}a&a\\a&a\end{pmatrix},\qquad
A=\begin{pmatrix}a&b&c\\ b&c&a\\ c&a&b\end{pmatrix}$$

Indeed, the $p=2$ assertion is clear. Regarding now the $p=3$ assertion, consider an arbitary $3\times 3$ bistochastic matrix, written as follows:
$$A=\begin{pmatrix}a&b&n-a-b\\ d&c&n-c-d\\ n-a-d&n-b-c&*\end{pmatrix}$$

Here $*=a+b+c+d-n$, but we won't use this value, because one of the 3 diagonal equations is redundant anyway. With these notations in hand, the conditions are:
$$b+(n-c-d)+(n-a-d)=n$$
$$(n-a-b)+d+(n-b-c)=n$$

Now since substracting these equations gives $b=d$, we obtain the result.

Regarding now the last assertion, consider the following $p\times p$ permutation matrix:
$$A=\begin{pmatrix}
1&&&&\\ 
&&&&1\\
&&&1\\
&&\iddots\\
&1
\end{pmatrix}$$

Since this matrix is tristochastic, for any $p\geq 3$ odd, this gives the result.
\end{proof}

As already mentioned, it is not clear on how to understand the tristochastic matrices at $p\geq 5$. Such matrices make sense of course at any $p\geq 2$, not necessarily prime, so we can try to first study them at $p=4$. But, the situation at $p=4$ is not clear either.

Regarding now the asymptotic count, we have here:

\begin{theorem}
For $p=2,3$, the probability for a randomly chosen $M\in M_{3\times N}(\mathbb Z_p)$, with $N\in p\mathbb N$, $N\to\infty$, to be partial Butson is respectively:
$$P_3^{(2)}\simeq\begin{cases}
\frac{16}{\sqrt{(2\pi N)^3}}&{\rm if}\ N\in4\mathbb N\\
0&{\rm if}\ N\notin 4\mathbb N\end{cases}\qquad\qquad
P_3^{(3)}\simeq\frac{243\sqrt{3}}{(2\pi N)^3}$$
In addition, we have $P_3^{(p)}>0$ for any $N\in p\mathbb N$, for any $p\geq 3$ prime.
\end{theorem}

\begin{proof}
According to Proposition 4.2, and then to the Stirling formula, we have:
$$P_3^{(2)}=\frac{1}{4^N}\binom{N}{N/4,N/4,N/4,N/4}\simeq\frac{16}{\sqrt{(2\pi N)^3}}$$

Also, by using Proposition 4.2, and then Lemma 2.4 with $s=p=3$, $n=N/3$:
\begin{eqnarray*}
P_3^{(3)}
&=&\frac{1}{9^N}\sum_{a+b+c=N/3}\binom{N}{a,b,c,b,c,a,c,a,b}\\
&=&\frac{1}{3^N}\binom{N}{N/3,N/3,N/3}\times\frac{1}{3^N}\sum_{a+b+c=N/3}\binom{N/3}{a,b,c}^3\\
&\simeq&\frac{3\sqrt{3}}{2\pi N}\cdot\sqrt{\frac{81}{(2\pi N/3)^4}}=\frac{243\sqrt{3}}{(2\pi N)^3}
\end{eqnarray*}

Finally, the last assertion is clear from the last assertion in Proposition 4.2.
\end{proof}

For exponents of type $q=p^k$ with $k\geq 2$, we obtain as well some kind of tristochastic matrices, but with the tristochastic condition taken blockwise. However, finding a good parametrization of these matrices is quite unobvious, even at $q=4$.

\section{Further results, conclusion}

We have the following question, which emerges from the above results:

\begin{question}
Is there any de Launey-Levin type formula for $P_M$, with $M\in\mathbb N$ arbitrary, at least in the case where $q=p$ is prime?
\end{question}

As a first observation, the beginning of the proof in \cite{ll2} applies to the general situation $q\in\mathbb N$. Indeed, by following the idea there, we have:

\begin{theorem}
The probability $P_M$ for a random $H\in M_{M\times N}(\mathbb Z_q)$ to be partial Butson equals the probability for a length $N$ random walk with increments drawn from
$$E=\left\{(e_i\bar{e}_j)_{i<j}\Big|e\in\mathbb Z_q^M\right\}$$
regarded as a subset $\mathbb Z_q^{\binom{M}{2}}$, to return at the origin.
\end{theorem}

\begin{proof}
Indeed, with $T(e)=(e_i\bar{e}_j)_{i<j}$, a matrix $X=[e_1,\ldots,e_N]\in M_{M\times N}(\mathbb Z_q)$ is partial Butson if and only if $T(e_1)+\ldots+T(e_N)=0$, and this gives the result.
\end{proof}

Observe now that, according to the above result, we have:
$$P_M
=\frac{1}{q^{(M-1)N}}\#\left\{\xi_1,\ldots,\xi_N\in E\Big|\sum_i\xi_i=0\right\}
=\frac{1}{q^{(M-1)N}}\sum_{\xi_1,\ldots,\xi_N\in E}\delta_{\Sigma\xi_i,0}$$

The problem is to continue the computation in the proof of the inversion formula. More precisely, the next step at $q=2$, which is the key one, is as follows:
$$\delta_{\Sigma\xi_i,0}=\frac{1}{(2\pi)^D}\int_{[-\pi,\pi]^D}e^{i<\lambda,\Sigma\xi_i>}d\lambda$$

Here $D=\binom{M}{2}$. The problem is that this formula works when $\Sigma\xi_i$ is real, as is the case in \cite{ll2}, but not when $\Sigma\xi_i$ is complex, as is the case in Theorem 5.2.

Yet another problem comes from the fact that the exponent $p=2$ used in \cite{ll2} is quite special, because it forces $N\in p^2\mathbb N$, instead of just $N\in p\mathbb N$. Thus, regardless of the above-mentioned real vs. complex issue, the combinatorics in \cite{ll2} is probably not exactly the $p=2$ instance of a ``generic'' combinatorics, because the generic case would probably require the assumption $p\neq 2$. We have no answer so far to these questions.

\end{document}